\documentclass[10pt,a4paper]{article}
\usepackage{amsthm,amssymb,mathrsfs,setspace,amsmath}
\usepackage[a4paper,bindingoffset=0.2in,
left=0.5in,right=0.5in,
top=1in,bottom=1in,footskip=.25in]{geometry}
\title{Linear canonical wavelet transform and the associated uncertainty principles}
\author{Bivek Gupta$^a$\thanks{bivekgupta040792@gmail.com}, Amit K. Verma$^b$\thanks{Corresponding author email: akverma@iitp.ac.in}, \\{\small\textit{$^{a,b}$ Department of Mathematics, IIT Patna, Bihta, Patna 801103.}}\\
Carlo Cattani$^c$\thanks{cattani@unitus.it}\\\small{\it{$^c$Engineering School (DEIM), University of Tuscia,}}\\\small{\it{
Largo dell'Universit\`a, 01100 Viterbo, Italy.}}}
\theoremstyle{definition}
\newtheorem{definition}{Definition}[section]
\newtheorem{lemma}{Lemma}[section]
\newtheorem{proposition}{Proposition}[section]
\newtheorem{theorem}{Theorem}[section]
\newtheorem{corollary}{Corollary}[section]
\newtheorem{remark}{Remark}[section]
\usepackage{hyperref}
\usepackage[numbers]{natbib}
\usepackage{notoccite}

\begin{document}
\maketitle
\begin{abstract}
We define a novel time-frequency analyzing tool, namely linear canonical wavelet transform (LCWT) and study some of its important properties like inner product relation, reconstruction formula and also characterize its range. We obtain Donoho-Stark's and Lieb's uncertainty principle for the LCWT and give a lower bound for the measure of its essential support. We also give Shapiro's mean dispersion theorem for the proposed LCWT.
\end{abstract}
\section{Introduction}
We first mention below some important abbreviations that will be used throughout this paper.\\~\\
\textit{List of Abbreviations}\\
FT - Fourier transform\\
FrFT - Fractional Fourier transform\\
LCT - Linear canonical transform\\
WT - Wavelet transform\\
FrWT - Fractional wavelet transform\\
WLCT - Windowed linear canonical transform\\
LCWT - Linear canonical wavelet transform\\
ONS - Orthonormal sequence\\
RKHS - Reproducing kernel Hilbert Space\\
UP - Uncertainty principle\\~\\

As a generalization of FT \cite{debnath2014integral} and FrFT \cite{almeida1994fractional, CHEN202171, i1998fractional}, LCT is a four parameter family of linear integral transform proposed by Mohinsky and Quesne \cite{moshinsky1971linear} and is considered as the important tool for non-stationary signal processing. Because of the extra degrees of freedom, as compared to the FT and FrFT, its application can be found in number of fields including signal separation \cite{sharma2006signal}, signal reconstruction \cite{wei2014reconstruction}, filter designing \cite{barshan1997optimal} and many more. Recently in \cite{GAO2021108233} the authos studied octonion linear canonical transform. For more detail on LCT and its application we refer the reader to \cite{healy2015linear}.

Even though the wavelet transform (WT)\cite{debnath2002wavelet} is a potential tool for the analysis of non-stationary signals, it is incompetent for analyzing the signals with not well concentrated energy in the time-frequency plane, for example, the chirp like signal which is ubiquitous in nature \cite{dai2017new}. On the other hand, for the signal whose energy in the frequency domain is not well concentrated, LCT is an 
appropriate tool. However, because of its global kernel it is not capable of indicating the time localization of the LCT spectral components, and thus LCT is not suitable for processing the non-stationary signal whose LCT spectral characteristics changes with time. The WLCT \cite{kou2012windowed} is thus proposed to overcome this drawback. In this case, the original signal is first segmented with time localization window, followed by performing the LCT spectral analysis for these segment. WLCT is capable of offering a joint signal representation in both time and LCT domain, but its fixed window width limits the practical application, it is impossible to provide good time resolution and spectral resolution simultaneously.

Thus to circumvent these limitations of LCT, WT and WLCT we propose a novel LCWT. In fact, Wei et al. \cite{wei2014generalized} and Guo et al. \cite{guo2018linear} generalized the FrWT, studied in \cite{shi2012novel}, to the LCWT.  Wei et al. \cite{wei2014generalized} studied its resolution in time and linear canonical domains and Guo et al. \cite{guo2018linear} studied its properties on Sobolev spaces. Dai et al. \cite{dai2017new} gave a new definition of the FrWT(also see \cite{prasad2014generalized}), which we generalize in the context of the LCT and study the associated UP.

In Harmonic analysis, the UP is a relation between a function and its FT which says that a function (non-zero) and its FT cannot be very well localize simultaneously. This general fact is interpreted in several different ways, for this we refer the reader to a survey paper by Folland and Sitaram \cite{folland1997uncertainty}. Shapiro, in \cite{shapiro1991uncertainty}, studied the localization for an ONS and proved that if an ONS $\{\phi_{k}\}$ in $L^2(\mathbb{R})$ and the sequence of their FT $\{\hat{\phi}_{k}\}$ are such that their means and dispersions are uniformly bounded, then $\{\phi_{k}\}$ is finite. Jaming and Powell \cite{jaming2007uncertainty} proved a quantitative version of Shapiro's theorem which says that for an ONS $\{\phi_{k}\}$ in $L^2(\mathbb{R})$ and $N\in\mathbb{N},$
$$\sum_{k=1}^{N}\left(\|t\phi_{k}\|_{L^2(\mathbb{R})}^2+\|\xi\hat{\phi}_{k}\|_{L^2(\mathbb{R})}^2\right)\geq\frac{(N+1)^2}{2\pi}.$$

A  multivariable quantitative version of Shapiro's theorem for generalized dispersion was proved by Malinnikova  \cite{malinnikova2010orthonormal}. It states that if $\{\phi_{k}\}$ be an ONS in $L^2(\mathbb{R}^d),$ $N\in\mathbb{N}$ and $p>0$ then $\exists$ $C_{p,d}$ for which

$$\sum_{k=1}^{N}\left(\||t|^{\frac{p}{2}}\phi_{k}\|_{L^2(\mathbb{R}^d)}^2+\||\xi|^{\frac{p}{2}}\hat{\phi}_{k}\|_{L^2(\mathbb{R}^d)}^2\right)\geq C_{p,d}N^{1+\frac{p}{2d}}.$$

Recently, in this direction Shapiro's mean dispersion theorem has been proved for many integral transforms like short-time FT \cite{lamouchi2016time}, WT  \cite{hamadi2017shapiro}, Hankel WT \cite{b2018uncertainty}, Hankel Stockwell transform \cite{hamadi2020uncertainty}, shearlet transform
\cite{nefzi2021shapiro}, etc.

The main objectives of this paper are 
(i) to define a  new time-frequency analysing tool, namely LCWT, which generalizes the FrWT studied in \cite{dai2017new} in the context of LCT, and study some of its basic properties along with the inner product relation, reconstruction formula and also characterize its range
(ii) to study the time-LCT frequency analysis and the associated constant $Q-$factor
(iii) to establish an UP for the LCWT for a finite energy signal. The UP for the LCWT can be derived from the UP of the LCT following the strategy adopted by Wilczok \cite{wilczok2000new} and Verma et al. \cite{verma2021note} while deriving the UP for the WT and the FrWT respectively. Similar UP has been introduced for several integral transform like fractional WT \cite{bahri2017logarithmic}, non-isotropic angular Stockwell transform \cite{shah2019non}, etc. However, we are interested in proposing an uncertainty  principle directly for the LCWT  without using the UP associated with LCT. In this regard we establish the Donoho-Stark's and the Lieb's UP for the LCWT, which in turn provide a lower bound for the measure of essential support of the LCWT. See also \cite{kou2012windowed},\cite{huo2019uncertainty}, for similar results in case of other integral transforms.
(iv) to study the Shapiro's mean dispersion theorem for the LCWT.

The paper is arranged as follows. In section 2, we recall the definition of LCT and some of its properties. In section 3, we define LCWT and study some of its basic properties including inner product relation, reconstruction formula and also characterize the range of the transform. Donoho-Stark's and Lieb's UP for the proposed LCWT are studied in section 4. Section 5 is devoted  to Shapiro's mean dispersion theorem for LCWT. Finally, in section 6, we conclude our paper.
\section{Preliminaries}
We briefly recall the definition of LCT and its important properties that we will be using in the sequel.
\subsection{LCT}
\begin{definition}
The LCT of $f\in L^2(\mathbb{R})$, with respect to a matrix parameter $$M=
\begin{bmatrix}
A & B\\
C & D
\end{bmatrix}, AD-BC=1$$
is defined as 
$$(\mathcal{L}^{M}f)(\xi)=\begin{cases}
\displaystyle\int_{\mathbb{R}}f(t)K_{M}(t,\xi)d t, B\neq 0,\\
\sqrt{D}e^{\frac{i}{2}CD\xi^2}f(D\xi),~B=0,
\end{cases}$$
where $K_{M}(t,\xi)$ is a kernel given by
\begin{equation}
K_{M}(t,\xi)=
 \frac{1}{\sqrt{2\pi iB}} e^{\frac{i}{2}\left(\frac{A}{B}t^2-\frac{2}{B}\xi t+\frac{D}{B}\xi^2\right)},~\xi\in\mathbb{R}.
\end{equation}
\end{definition}
Among several important properties of the LCT, the important among them, that will be used in the sequel, is the Parseval's formula
\begin{equation}\label{P3ParsevalLCT}
\int_{\mathbb{R}}f(t)\overline{g(t)}dt=\int_{\mathbb{R}}(\mathcal{L}^Mf)(\xi)\overline{(\mathcal{L}^Mg)(\xi)}d\xi,~\mbox{where}~f,~g\in L^2(\mathbb{R}).
\end{equation}
Particularly, if $f=g,$ then we have the Plancherel's formula
\begin{equation}\|f\|_{L^2(\mathbb{R})}=\|\mathcal{L}^Mf\|_{L^2(\mathbb{R})}.
\end{equation}
The LCTs satisfies the additive property, i.e.,
\begin{equation}
\mathcal{L}^M\mathcal{L}^Nf=\mathcal{L}^{MN}f,~\mbox{where} ~f\in L^2(\mathbb{R}),
\end{equation}
and the inversion property
\begin{equation}
\mathcal{L}^{M^{-1}}\left(\mathcal{L}^{M}f\right)=f,
\end{equation}
where, $M^{-1}$ denotes the inverse of $M.$ For convenient, we now denote the matrix $M$ by $(A,B;C,D).$
\section{LCWT}
We propose a new integral transform namely the LCWT. This definition is mainly motivated from the definition of FrWT defined by Dai et al. \cite{dai2017new}. We shall discuss some of its basic properties along with the inner product relation, reconstruction formula and also prove that its range is a RKHS.

Motivated by the definition of the admissible wavelet pair in \cite{daubechies1992ten}, we first define it in the setting of LCT domain.
\begin{definition}
A pair $\{\psi,\phi\}$ of functions in $L^2(\mathbb{R})$ is said to be an admissible linear canonical wavelet pair (ALCWP) if they satisfy the following admissibility condition
\begin{equation}\label{P3ACPair}
C_{\psi,\phi,M}:=\int_{\mathbb{R^+}}\overline{(\mathcal{L}^M\psi)\left(\frac{\xi}{a}\right)}(\mathcal{L}^M\phi)\left(\frac{\xi}{a}\right)\frac{da}{a}
\end{equation}
is a non-zero complex constant independent of $\xi\in\mathbb{R}$ satisfying $|\xi|=1.$
In case $\psi=\phi,$ we denote $C_{\psi,\psi,M}$ by $C_{\psi,M}$ and the required admissibility condition reduces to 
\begin{equation}\label{P3AC}
 C_{\psi,M}:=\int_{\mathbb{R^+}}\left|(\mathcal{L}^M\psi)\left(\frac{\xi}{a}\right)\right|^2\frac{da}{a}
\end{equation}
is a positive constant independent of $\xi$ satisfying $|\xi|=1.$ We call $\psi\in L^2(\mathbb{R}),$ satisfying equation (\ref{P3AC}), the admissible linear canonical wavelet (ALCW).
\end{definition}
We now give the definition of the novel LCWT.
\begin{definition}
Let $f\in L^2(\mathbb{R}),$ $M=(A,B;C,D)$ be a matrix with $AD-BC=1~\mbox{and}~B\neq 0$ then we define the LCWT of $f$ with respect to $M$ and an ALCW $\psi$ by
$$(W^M_{\psi}f)(a,b)=e^{-\frac{iA}{2B}b^2}\left\{f(t)e^{\frac{iA}{2B}t^2}\star\overline{\sqrt{a}\psi(-at)e^{\frac{iA}{2B}(at)^2}}\right\}(b),~a\in\mathbb{R^+},b\in\mathbb{R},$$
\end{definition}
where $\star$ denotes the convolution given by
$$(f\star g)(\nu)=\int_{\mathbb{R}}f(x)g(\nu-x)dx,~\nu\in\mathbb{R}.$$
Equivalently, 
\begin{eqnarray*}
\left(W^M_\psi f\right)(a,b)&=&e^{-\frac{iA}{2B}b^2}\int_{\mathbb{R}}f(t)e^{\frac{iA}{2B}t^2}\overline{\sqrt{a}\psi(-a(b-t))e^{\frac{iA}{2B}(a(t-b))^2}}dt\\
&=&\int_{\mathbb{R}}f(t)\overline{e^{-\frac{iA}{2B}\{(t^2-b^2)-(a(t-b))^2\}}\sqrt{a}\psi(a(t-b))}dt\\
&=&\int_{\mathbb{R}}f(t)\overline{\psi^M_{a,b}(t)}dt,
\end{eqnarray*}
where, with $\psi_{a,b}(t)=\sqrt{a}\psi(a(t-b))$ 
\begin{equation}\label{P3DW}
\psi^M_{a,b}(t)=e^{-\frac{iA}{2B}\{(t^2-b^2)-(a(t-b))^2\}}\psi_{a,b}(t).
\end{equation}
Thus, we have an equivalent definition of the LCWT  as 
\begin{equation}\label{P3WTDef}
(W^M_{\psi}f)(a,b)=\langle f,\psi^M_{a,b} \rangle_{L^2(\mathbb{R})}.
\end{equation}
It is to be noted that depending on the different choice of the matrix $M,$ we have different integral transforms:
\begin{itemize}
\item For $M=(\cos\alpha,\sin\alpha;-\sin\alpha,\cos\alpha),\alpha\neq n\pi,$ we obtain the FrWT as discussed in \cite{dai2017new}.
\item For $M=(0,1;-1,0)$ we obtain the traditional WT \cite{daubechies1992ten}.
\end{itemize}
We now establish a fundamental relation between LCWT and the LCT. This relation will be useful in obtaining the resolution of time and linear canonical spectrum in the time-LCT-frequency plane and inner product relation associated with the LCWT. 
\begin{proposition}
If $W_{\psi}^Mf$ and $\mathcal{L}^{M}f$ are respectively the LCWT and the LCT of $f\in L^2(\mathbb{R}).$ Then,
\begin{equation}\label{P3LCTWT}
\mathcal{L}^M\left((W_{\psi}^Mf)(a,\cdot)\right)(\xi)=\frac{\sqrt{-2\pi iB}}{\sqrt{a}}e^{\frac{iD}{2B}(\frac{\xi}{a})^2}(\mathcal{L}^M f)(\xi)\overline{(\mathcal{L}^M \psi)\left(\frac{\xi}{a}\right)}.
\end{equation}
\end{proposition}
\begin{proof}
Form the definition of the LCT and $\psi_{a,b}^M,$ it follows that
\begin{eqnarray*}
\left(\mathcal{L}^M\psi_{a,b}^M\right)(\xi)&=&\int_{\mathbb{R}}\sqrt{a}\psi(a(t-b))\sqrt{\frac{1}{2\pi iB}}e^{\frac{i}{2}\left\{\frac{Ab^2}{B}+\frac{A}{B}(a(t-b))^2-\frac{2}{B}\xi t+\frac{D}{B}\xi^2\right\}}dt\\
&=& \int_{\mathbb{R}}\sqrt{a}\psi(at)\sqrt{\frac{1}{2\pi iB}}e^{\frac{i}{2}\left\{\frac{Ab^2}{B}+\frac{A}{B}(at)^2-\frac{2}{Ba}(at+ab)\xi +\frac{D}{B}\xi^2\right\}}dt\\
&=&\frac{1}{\sqrt{a}}\int_{\mathbb{R}}\psi(t)\sqrt{\frac{1}{2\pi iB}}e^{\frac{i}{2}\left(\frac{Ab^2}{B}-\frac{2}{B}\xi b+\frac{D}{B}\xi^2\right)}e^{\frac{i}{2}\left(\frac{Ab^2}{B}-\frac{2}{B}t\left(\frac{\xi}{a}\right)+\frac{D}{B}\left(\frac{\xi}{a}\right)^2\right)}e^{\frac{-iD}{2B}\left(\frac{\xi}{a}\right)^2}dt\\
&=&\frac{1}{\sqrt{a}}e^{\frac{-iD}{2B}\left(\frac{\xi}{a}\right)^2}\sqrt{2\pi iB}\int_{\mathbb{R}}\psi(t)K_{M}(b,\xi)K_{M}\left(t,\frac{\xi}{a}\right)dt.
\end{eqnarray*}
Therefore, we have
\begin{equation}\label{P3LCTDW}
\left(\mathcal{L}^M\psi_{a,b}^M\right)(\xi)=\frac{\sqrt{2\pi iB}}{\sqrt{a}}K_{M}(b,\xi)\left(\mathcal{L}^M\psi\right)\left(\frac{\xi}{a}\right).
\end{equation}
Using (\ref{P3ParsevalLCT}) in (\ref{P3WTDef}), we get 
$$(W_{\psi}^Mf)(a,b)=\langle \mathcal{L}^Mf,\mathcal{L}^M\left(\psi_{a,b}^M\right) \rangle_{L^2(\mathbb{R})}.$$
Using equation (\ref{P3LCTDW}), we have
\begin{equation}\label{P3WTLCD}
(W_{\psi}^Mf)(a,b)=\frac{\sqrt{-2\pi iB}}{\sqrt{a}}\int_{\mathbb{R}}e^{\frac{-iD}{2B}\left(\frac{\xi}{a}\right)^2}\left(\mathcal{L}^Mf\right)(\xi)\overline{\left(\mathcal{L}^M\psi\right)\left(\frac{\xi}{a}\right)}K_{M^{-1}}(b,\xi)d\xi.
\end{equation}
Therefore, it follows that 
$$\mathcal{L}^M\left((W_{\psi}^Mf)(a,\cdot)\right)(\xi)=\frac{\sqrt{-2\pi iB}}{\sqrt{a}}e^{\frac{iD}{2B}(\frac{\xi}{a})^2}(\mathcal{L}^M f)(\xi)\overline{(\mathcal{L}^M \psi)\left(\frac{\xi}{a}\right)}.$$
This completes the proof.
\end{proof}
\subsection{Time-LCT frequency analysis}
From equation (\ref{P3WTDef}) it follows that if $\psi^M_{a,b}$ is supported in the time domain, then so is $(W^M_\psi f)(a,b).$
Also, from equation (\ref{P3WTLCD}),
%$$(W^M_\psi f)(a,b)=\frac{\sqrt{-2\pi iB}}{\sqrt{a}}\int_\mathbb{R}e^{\frac{iD}{2B}(\frac{\xi}{a})^2}(\mathcal{L}^M f)(\xi)\overline{(\mathcal{L}^M \psi)\left(\frac{\xi}{a}\right)}K_{M^{-1}}(b,\xi)d\xi$$ 
it follows that the LCWT can provide the local property of $f(t)$ in the linear canonical domain. Thus the LCWT is capable of producing simultaneously the time-LCT frequency information and represent the signal in the time-LCT frequency domain. More precisely, if $\psi$ and $\mathcal{L}^M \psi$ are window functions in time and linear canonical domain respectively with $E_\psi$ and $E_{\mathcal{L}^M \psi}$ as centres and $\Delta_{\psi}$ and $\Delta_{\mathcal{L}^M \psi}$ are radii respectively. Then the centre and radius of $\psi^M_{a,b}$ are given respectively by 
$$E[\psi^M_{a,b}]=\frac{1}{a}E_\psi +b,$$ 
and
$$\Delta[\psi^M_{a,b}]=\frac{1}{a}\Delta_\psi.$$
Similarly, the centre and radius of window function $(\mathcal{L}^M\psi)\left(\frac{\xi}{a}\right)$ are given by 
$$E\left[(\mathcal{L}^M\psi)\left(\frac{\xi}{a}\right)\right]=aE_{\mathcal{L^M\psi}},$$
and 
$$\Delta\left[(\mathcal{L}^M\psi)\left(\frac{\xi}{a}\right)\right]=a\Delta_{\mathcal{L^M\psi}}.$$
Thus, the $Q$-factor of the window function of the linear canonical transform domain is
$$Q=\frac{\Delta_{\mathcal{L^M\psi}}}{E_{\mathcal{L^M\psi}}},$$
which is independent of the scaling parameter $a$ for a given parameter $M.$  This is called the constant $Q-$property of the LCWT.
\subsection{Time-LCT frequency resolution}
The LCWT $(W_\psi^M f)(a,b)$ localizes the signal $f$ in the time window 
$$\left[\frac{1}{a}E_\psi+b-\frac{1}{a}\Delta_
\psi,\frac{1}{a}E_\psi+b+\frac{1}{a}\Delta_
\psi\right].$$
Similarly, we get that the LCWT gives linear canonical spectrum content of $f$ in the window
$$\left[aE_{\mathcal{L}^M\psi}-a\Delta_{\mathcal{L}^M\psi},aE_{\mathcal{L}^M\psi}+a\Delta_{\mathcal{L}^M\psi}\right].$$
Thus, the joint resolution of the LCWT in the time and linear canonical domain is given by the window 
$$\left[\frac{1}{a}E_\psi+b-\frac{1}{a}\Delta_
\psi,\frac{1}{a}E_\psi+b+\frac{1}{a}\Delta_
\psi\right]\times\left[aE_{\mathcal{L}^M\psi}-a\Delta_{\mathcal{L}^M\psi},aE_{\mathcal{L}^M\psi}+a\Delta_{\mathcal{L}^M\psi}\right],$$
with constant area $4\Delta_{\psi}\Delta_{\mathcal{L}^M\psi}$ in the time-LCT-frequency plane.
Thus it follows that for a given parameter $M,$ the window area depends on the linear canonical admissible wavelets and is independent of the parameters $a$ and $b.$ But it is to be noted that the the window gets narrower for large value of $a$ and wider for small value of $a.$ Thus the window given by the transform is flexible and hence, it is capable of simultaneously providing the time linear canonical domain information. This flexibility of the window makes the proposed LCWT more advantageous then the WLCT as in this case the window is rigid.

Some basic properties of LCWT is given below.
\begin{theorem}
Let $g,h\in L^2(\mathbb{R}),$ $\psi$ and $\phi$ are ALCWs, $\alpha,\beta\in\mathbb{C}$, $\lambda> 0$ and $y\in \mathbb{R}.$ Then
\begin{enumerate}
\item $W^M_{\psi}(\alpha g+\beta h)=\alpha (W^M_{\psi}g)+\beta (W^M_{\psi}h)$
\item $W^M_{\alpha \psi+\beta\phi}(g)=\bar\alpha (W^M_{\psi}g)+\bar\beta (W^M_{\phi}g)$
\item $(W^M_{\psi}\delta_{\lambda} g)(a,b)=(W^{\tilde{M}}_{\psi}g)(\frac{a}{\lambda},b\lambda),~\mbox{where}~ (\delta_\lambda g)(t)=\sqrt{\lambda}g(\lambda t)$ and $\tilde{M}=(A,{\lambda}^2B;\frac{C}{{\lambda}^2},D)$
\item $(W^M_{\psi}\tau_{y} g)(a,b)=e^{\frac{iA}{B}y(y-b)}(W^{M}_{\psi} e^{\frac{iA}{B}yt}g)(a,b-y),~\mbox{where}~ (\tau_{y}g)(t)=g(t-y)$.
\end{enumerate}
\end{theorem}
\begin{proof}
The proofs are immediate and can be omitted.
\end{proof}
If $\{\psi,\phi\}$ is admissible linear canonical wavelet pair such that each $\phi$ and $\psi$ are ALCWs and $f,g\in L^2(\mathbb{R})$ and is such that they are orthogonal then $W_{\psi}^M(f)$ and $W_{\phi}^M(g)$ are orthogonal in $L^2(\mathbb{R}^+\times\mathbb{R},dadb).$ This result is justified by the following theorem, which further gives the resolution of identity for the LCWT. 
\begin{theorem}(\textbf{Inner product relation for LCWT}).\label{P3theo2.2}
 Let $\{\psi,\phi\}$ be an ALCWP such that $\psi$ and $\phi$ are ALCWs and $f,g\in L^2(\mathbb{R}),$ then
\begin{equation}\label{P3IPR}
\langle W^M_{\psi}f ,W^M_{\phi}g \rangle_{L^2(\mathbb{R^+}\times\mathbb{R})}=2\pi|B|C_{\psi,\phi,M}\langle f,g\rangle_{L^2(\mathbb{R})},
\end{equation}
where $C_{\psi,\phi,M}$ is provided in (\ref{P3ACPair}).
\end{theorem}
\begin{proof}
Using equation (\ref{P3LCTWT}), we get
\begin{eqnarray*}
\langle W^M_\psi f,W^M_\phi g \rangle_{L^2(\mathbb{R^+}\times\mathbb{R})}&=&\int_{\mathbb{R^+}\times\mathbb{R}}\left(W^M_\psi f\right)(a,b)\overline{\left(W^M_\phi g\right)(a,b)}dadb\\
&=&\int_{\mathbb{R^+}\times\mathbb{R}}\left(\mathcal{L}^M\left(W^M_\psi f\right)(a,\cdot)\right)(\xi)\overline{\left(\mathcal{L}^M\left(W^M_\phi g\right)(a,\cdot)\right)(\xi)}d\xi da\\
&=&\int_{\mathbb{R^+}\times\mathbb{R}}\frac{2\pi|B|}{a}\left(\mathcal{L}^Mf\right)(\xi) \overline{\left(\mathcal{L}^M\psi\right)\left(\frac{\xi}{a}\right)}\overline{\left(\mathcal{L}^Mg\right)(\xi)}\left(\mathcal{L}^M\phi\right)\left(\frac{\xi}{a}\right)d\xi da\\
&=& 2\pi|B|\int_{\mathbb{R}}\left(\mathcal{L}^Mf\right)(\xi)\overline{\left(\mathcal{L}^Mg\right)(\xi) }\left\{\int_{\mathbb{R}^+}\overline{\left(\mathcal{L}^M\psi\right)\left(\frac{\xi}{a}\right)}\left(\mathcal{L}^M\phi\right)\left(\frac{\xi}{a}\right)\frac{da}{a}\right\}d\xi \\
&=&2\pi|B|C_{\psi,\phi,M}\left\langle\mathcal{L}^Mf,\mathcal{L}^Mg \right\rangle_{L^2(\mathbb{R})}\\
&=&2\pi|B|C_{\psi,\phi,M}\langle f,g\rangle.
\end{eqnarray*}
\end{proof}
\begin{remark}
Replacing $\psi=\phi$ in equation (\ref{P3IPR}), we have 
\begin{equation}\label{P3IPRWT}
\langle W^M_{\psi}f ,W^M_{\psi}g \rangle_{L^2(\mathbb{R^+}\times\mathbb{R})}=2\pi|B|C_{\psi,M}\langle f,g\rangle_{L^2(\mathbb{R})},
\end{equation}
where $C_{\psi,M}$ is provided in (\ref{P3AC}).
\end{remark}
\begin{remark}
(Plancherel's theorem for $W^M_{\psi}$) Replacing $f=g$ and $\phi=\psi$ in equation (\ref{P3IPR}) we have the Plancherel's theorem for $W_\psi^M$ given by
\begin{equation}\label{P3PlancherelWT}
\|W^M_{\psi} f\|_{L^2(\mathbb{R^+}\times\mathbb{R})}=\left(2\pi|B|C_{\psi,M}\right)^{\frac{1}{2}}\|f\|_{L^2(\mathbb{R})}.
\end{equation}
Thus, from equation (\ref{P3PlancherelWT}), it follows that LCWT from $L^2(\mathbb{R})$ into $L^2(\mathbb{R^+\times\mathbb{R}})$ is a continuous linear operator. If further ALCW $\psi$ is such that $C_{\psi,M}=\frac{1}{2\pi|B|},$ then the operator is an isometry.
\end{remark}
\begin{theorem}(\textbf{Reconstruction formula}).\label{P3theo2.3}
 Let $\{\psi,\phi\}$ be an ALCWP such that $\psi$ and $\phi$ are ALCWs and $f\in L^2(\mathbb{R})$, then $f$ can be given by the formula
\begin{equation}
f(t)=\frac{1}{2\pi|B|C_{\psi,\phi,M}}\int_{\mathbb{R^+}\times\mathbb{R}}(W^M_{\psi}f)(a,b)\phi^M_{a,b}(t)dadb~~\mbox{a.e.}~t\in\mathbb{R}.
\end{equation}
\end{theorem}
\begin{proof}
From equation (\ref{P3IPR}), we get
\begin{eqnarray*}
2\pi|B|C_{\psi,\phi,M}\langle f,g\rangle_{L^2(\mathbb{R})}&=&\langle W^M_\psi f,W^M_\phi g \rangle_{L^2(\mathbb{R^+}\times\mathbb{R})}\\
&=&\int_{\mathbb{R}^+\times\mathbb{R}}\left(W_\psi^M f\right)(a,b)\overline{\left(\int_{\mathbb{R}}g(t)\overline{\phi_{a,b}^M(t)}dt\right)}dadb\\
%&=&\int_{\mathbb{R}}\left(\int_{\mathbb{R}^+\times\mathbb{R}}\left(W_\psi^M f\right)(a,b)\phi_{a,b}^M(t)dadb\right)\overline{g(t)}dt\\
&=&\left\langle\int_{\mathbb{R}^+\times\mathbb{R}}\left(W_\psi^M f\right)(a,b)\phi_{a,b}^M(t)dadb,g(t)\right\rangle_{L^2(\mathbb{R})}.
\end{eqnarray*}
Since $g\in L^2(\mathbb{R})$ is arbitrary, we have
$$f(t)=\frac{1}{2\pi|B|C_{\psi,\phi,M}}\int_{\mathbb{R}^+\times\mathbb{R}}\left(W_\psi^M f\right)(a,b)\phi_{a,b}^M(t)dadb~\mbox{a.e.}$$
The proof is complete.
\end{proof}
In particular, if $\psi=\phi$ then we have the following reconstruction formula
$$f(t)=\frac{1}{2\pi|B|C_{\psi,M}}\int_{\mathbb{R}^+\times\mathbb{R}}\left(W_\psi^M f\right)(a,b)\psi_{a,b}^M(t)dadb~\mbox{a.e.}~t\in\mathbb{R}$$
The following theorem characterizes the range of the LCWT and proves that the range is a RKHS. It also gives the explicit expression for the reproducing kernel.
\begin{theorem}
For $\psi$ being ALCW, $W^M_\psi(L^2(\mathbb{R}))$ is a RKHS with the  kernel  
$$K^M_\psi(x,y;a,b)=\frac{1}{2\pi|B|C_{\psi,M}}\langle \psi^M_{a,b},\psi^M_{x,y}\rangle_{L^2(\mathbb{R})}, (x,y), (a,b)\in \mathbb{R^+\times\mathbb{R}}.$$
Moreover, the kernel is such that $|K^M_\psi(x,y;a,b)|\leq \frac{1}{2\pi|B|C_{\psi,M}}\|\psi\|^2_{L^2(\mathbb{R})}.$
\end{theorem}
\begin{proof}
For $(a,b)\in\mathbb{R}^+\times\mathbb{R},$ we see that 
$$K^M_\psi(x,y;a,b)=\frac{1}{2\pi|B|C_{\psi,M}}\left(W_\psi^M\psi_{a,b}^M\right)(x,y)~\mbox{for all}~(x,y)\in\mathbb{R}^+\times\mathbb{R}.$$
Now,
\begin{eqnarray*}
\|K^M_\psi(\cdot,\cdot;a,b)\|^2_{L^2(\mathbb{R}^+\times\mathbb{R})}&=&\frac{1}{2\pi|B|C_{\psi,M}}\|W_\psi^M\psi_{a,b}^M\|_{L^2(\mathbb{R}^+\times\mathbb{R})}\\
&=&\frac{1}{2\pi|B|C_{\psi,M}}\|\psi\|^2_{L^2(\mathbb{R})}.
\end{eqnarray*}
Therefore, for $(a,b)\in\mathbb{R}^+\times\mathbb{R},$ $K^M_\psi(x,y;a,b)\in L^2(\mathbb{R}^+\times\mathbb{R}).$
Now, let $f\in L^2(\mathbb{R})$
\begin{eqnarray*}
\left(W_\psi^Mf\right)(a,b)&=&\langle f,\psi_{a,b}^M\rangle_{L^2(\mathbb{R})}\\
&=&\frac{1}{2\pi|B|C_{\psi,M}}\langle W_\psi^M f,2\pi|B|C_{\psi,M} K_{\psi}^M(\cdot,\cdot;a,b)\rangle_{L^2(\mathbb{R}^+\times\mathbb{R})}\\
&=&\langle W_\psi^M f,K_{\psi}^M(\cdot,\cdot;a,b)\rangle_{L^2(\mathbb{R}^+\times\mathbb{R})}.
\end{eqnarray*} 
Thus, it follows that 
$$K^M_\psi(x,y;a,b)=\frac{1}{2\pi|B|C_{\psi,M}}\langle \psi^M_{a,b},\psi^M_{x,y}\rangle_{L^2(\mathbb{R})},$$ is the reproducing kernel of $W_\psi^M(L^2(\mathbb{R}).$

Again,
\begin{eqnarray*}
|K^M_\psi(x,y;a,b)|&=&\frac{1}{2\pi|B|C_{\psi,M}}|\langle \psi^M_{a,b},\psi^M_{x,y}\rangle_{L^2(\mathbb{R})}|\\
&\leq &\frac{1}{2\pi|B|C_{\psi,M}}\|\psi_{a,b}^M\|_{L^2(\mathbb{R})}\|\|\psi_{x,y}^M\|_{L^2(\mathbb{R})}\\
&=&\frac{\|\psi\|^2_{L^2(\mathbb{R})}}{2\pi|B|C_{\psi,M}}.
\end{eqnarray*}
This completes the proof.
\end{proof}
\section{Uncertainty principle}
We prove some UPs that limits the concentration of the LCWT in some subset in $\mathbb{R}^+\times\mathbb{R}$ of small measure. For related results in case of Fourier transform and windowed Fourier transform we refer the reader to \cite{donoho1989uncertainty},\cite{grochenig2001foundations}. Kou et al. \cite{kou2012paley} studied the same for the WLCT.

\begin{definition}
Let $0\leq\epsilon<1,$  $f\in L^2(\mathbb{R})$ and $E\subset\mathbb{R}$ be measurable, then $f$ is $\epsilon-$concentrated on $E$ if
 $$\left(\int_{E^{c}}|f(x)|^2dx\right)^\frac{1}{2}\leq\epsilon\|f\|_{L^2(\mathbb{R})}.$$
If $0\leq\epsilon\leq\frac{1}{2},$ then we say that most of the energy of $f$ is concentrated on $E$ and $E$ is called the essential support of $f$. If $\epsilon=0,$ then support of $f$ is contained in $E.$ 
 \end{definition}
\begin{lemma}
 If $\psi$ is an ALCW and $f\in L^2(\mathbb{R}).$ Then
$W^M_\psi f\in{L^p(\mathbb{R^+}\times\mathbb{R})},$ for all $p\in[2,\infty].$ Moreover,
\begin{equation}\label{P3LIp}
\|W^M_\psi f\|_{L^p(\mathbb{R^+}\times\mathbb{R})}\leq (2\pi|B|)^{\frac{1}{p}}C_{\psi,M}^{\frac{1}{p}}\|f\|_{L^2(\mathbb{R})}\|\psi\|^{1-\frac{2}{p}}_{L^2(\mathbb{R})},~p\in[2,\infty)
\end{equation}
\begin{equation}\label{P3Bound}
\|W^M_\psi f\|_{L^\infty(\mathbb{R^+}\times\mathbb{R})}\leq \|\psi\|_{L^2(\mathbb{R})}\|f\|_{L^2(\mathbb{R})}.
\end{equation}
 \end{lemma}
 \begin{proof}
 Since $\psi$ is an ALCW, it follows that $W^M_\psi f\in L^2(\mathbb{R^+}\times\mathbb{R}).$
 Again
 \begin{eqnarray*}
 \left|\left(W^M_\psi f\right)(a,b)\right| 
 &\leq &\|\psi\|_{L^2(\mathbb{R})}\|f\|_{L^2(\mathbb{R})}.
 \end{eqnarray*}
 Thus, $W^M_\psi f\in L^\infty(\mathbb{R^+}\times\mathbb{R}).$
 Also, since $W^M_\psi f\in L^2(\mathbb{R^+}\times\mathbb{R}),$ we have $W^M_\psi f\in L^p(\mathbb{R^+}\times\mathbb{R}),~p\in[2,\infty).$
 Moreover, 
 \begin{eqnarray*}
 \|W^M_\psi f\|_{L^p(\mathbb{R^+}\times\mathbb{R})} &\leq & \|W^M_\psi f\|^\frac{2}{p}_{L^2(\mathbb{R^+}\times\mathbb{R})} \|W^M_\psi f\|_{L^\infty(\mathbb{R^+}\times\mathbb{R})}^{1-\frac{2}{p}}\\
 &\leq & (2\pi|B|C_{\psi,M})^\frac{1}{p}\|f\|^\frac{2}{p}_{L^2(\mathbb{R})}\|f\|^{1-\frac{2}{p}}_{L^2(\mathbb{R})}\|\psi\|^{1-\frac{2}{p}}_{L^2(\mathbb{R})}.
\end{eqnarray*}
This proves the lemma.
 \end{proof}
 
 \begin{definition}
Let $0\leq\epsilon<1,$  $F\in L^2(\mathbb{R^+}\times\mathbb{R})$ and $\Omega\subset\mathbb{R^+}\times\mathbb{R}$ be measurable, then $F$ is $\epsilon-$concentrated on $\Omega$ if
 $$\left(\int_{\Omega^{c}}|F(x,y)|^2dxdy\right)^\frac{1}{2}\leq\epsilon\|F\|_{L^2(\mathbb{R^+}\times\mathbb{R})}.$$
If $0\leq\epsilon\leq\frac{1}{2},$ then we say that most of the energy of $F$ is concentrated on $\Omega$ and $\Omega$ is called the essential support of $F.$  If $\epsilon=0,$ then support of $F$ is contained in $\Omega.$ 
 \end{definition}
 We now prove the Donoho-Stark's UP for the propose LCWT.
 \begin{theorem}
Let $0\leq\epsilon<1,$ $\psi$ is an ALCW and a non-zero $f\in L^2(\mathbb{R}).$ Also let $W^M_{\psi} f$ is $\epsilon-$concentrated on $\Omega\subset\mathbb{R^+}\times\mathbb{R}$ then 
\begin{equation}
|\Omega|\|\psi\|^2_{L^2(\mathbb{R})}\geq 2\pi|B|C_{\psi,M}(1-\epsilon^2),
\end{equation}
where $|\Omega|$ denoted the measure of $\Omega.$
 \end{theorem}
 \begin{proof}
In equation (\ref{P3PlancherelWT}), we have
 $$\|W_\psi^Mf\|^2_{L^2(\mathbb{R^+}\times\mathbb{R})}=2\pi|B|C_{\psi,M}\|f\|^2_{L^2(\mathbb{R})}.$$
 Now, $$\int_{\mathbb{R^+}\times\mathbb{R}}|\left(W^M_{\psi}f\right)(a,b)|^2dadb\leq\int_{\mathbb{R^+}\times\mathbb{R}}\chi_\Omega(a,b)|\left(W^M_{\psi}f\right)(a,b)|^2dadb +\epsilon^2\|W_\psi^Mf\|_{L^2(\mathbb{R^+}\times\mathbb{R})}.$$
 This gives
 $$(1-\epsilon^2)\|W^M_\psi f\|_{L^2(\mathbb{R^+}\times\mathbb{R})}\leq|\Omega|\|W^M_\psi f\|^2_{L^\infty(\mathbb{R^+}\times\mathbb{R})}.$$
 Thus, using (\ref{P3Bound}), we get
 $$2\pi|B|C_{\psi,M}\|f\|^2_{L^2(\mathbb{R})}\leq |\Omega|\|f\|_{L^2(\mathbb{R})}\|\psi\|_{L^2(\mathbb{R})}.$$
The result follows, since $f\neq 0.$
 \end{proof}
\begin{corollary}{\label{P3D1}}
If $f\in L^2(\mathbb{R})\cap L^4(\mathbb{R}),$ in $L^2(\mathbb{R})-$norm, is $\epsilon_{E}-$concentrated on $E\subset\mathbb{R}$ and $W^M_{\psi}f$ is $\epsilon_{\Omega}-$concentrated on $\Omega\subset\mathbb{R^+}\times\mathbb{R},$ then 
$$|\Omega|m(E)\|\psi\|^2_{L^2(\mathbb{R})}\|f\|^4_{L^4(\mathbb{R})}\geq 2\pi|B|C_{\psi,M}(1-\epsilon_\Omega^2)(1-\epsilon_E^2)^2\|f\|^4_{L^2(\mathbb{R})},$$
where, $m(E)$ denotes the measure of $E.$
\end{corollary}
\begin{proof}
Since, $W^M_{\psi}f$ is $\epsilon_{\Omega}-$concentrated on $\Omega\subset\mathbb{R^+}\times\mathbb{R}$ in $L^2(\mathbb{R^+}\times\mathbb{R})-$norm, so we have
$|\Omega|\|\psi\|^2_{L^2(\mathbb{R})}\geq2\pi|B|C_{\psi,M}(1-\epsilon_\Omega^2).$
Again, since $f$ is $\epsilon_{E}-$concentrated, we have
$$\left(\int_{E^c}|f(x)|^2dx\right)^\frac{1}{2}\leq\epsilon_{E}\|f\|_{L^2(\mathbb{R})},$$ which further implies that
$$\|f\|_{L^2(\mathbb{R})}(1-\epsilon_E^2)\leq\int_{\mathbb{R}}\chi_{E}(x)|f(x)|^2dx.$$
We have by Holder's inequality  
$$\int_{\mathbb{R}}\chi_{E}(x)|f(x)|^2dx\leq\left(\int_{\mathbb{R}}|\chi_{E}(x)|^2dx\right)^\frac{1}{2}\|f\|^2_{L^4(\mathbb{R})}.$$
Thus
\begin{equation}
(1-\epsilon_{E}^2)\|f\|^2_{L^2(\mathbb{R})}\leq (m(E))^\frac{1}{2}\|f\|_{L^4(\mathbb{R})}.
\end{equation}
Therefore,
\begin{align*}
|\Omega|m(E)\|\psi\|^2_{L^2(\mathbb{R})}\|f\|^4_{L^4(\mathbb{R})}\geq 2\pi|B|C_{\psi,M}(1-\epsilon_\Omega^2)(1-\epsilon_E^2)^2\|f\|^4_{L^2(\mathbb{R})}.
\end{align*}
The proof is complete.
\end{proof}
\begin{corollary}{\label{P3D2}}
If $f\in L^2(\mathbb{R})\cap L^\infty(\mathbb{R}),$ in $L^2(\mathbb{R})-$norm, is $\epsilon_{E}-$concentrated on $E\subset\mathbb{R}$ and $W^M_{\psi}f$ is $\epsilon_{\Omega}-$concentrated on $\Omega\subset\mathbb{R^+}\times\mathbb{R},$ then 
$$|\Omega|m(E)\|\psi\|^2_{L^2(\mathbb{R})}\|f\|^2_{L^\infty(\mathbb{R})}\geq 2\pi|B|C_{\psi,M}(1-\epsilon_\Omega^2)(1-\epsilon_E^2)\|f\|^2_{L^2(\mathbb{R})}.$$
\end{corollary}
\begin{proof}
Since, $W^M_{\psi}f$ is $\epsilon_{\Omega}-$concentrated on $\Omega\subset\mathbb{R^+}\times\mathbb{R}$ in $L^2(\mathbb{R^+}\times\mathbb{R})-$norm, so we have
$|\Omega|\|\psi\|^2_{L^2(\mathbb{R})}\geq2\pi|B|C_{\psi,M}(1-\epsilon_\Omega^2).$
Again, since $f$ is $\epsilon_{E}-$concentrated, we have
$$\left(\int_{E^c}|f(x)|^2dx\right)^\frac{1}{2}\leq\epsilon_{E}\|f\|_{L^2(\mathbb{R})},$$ which further implies that
$$\|f\|_{L^2(\mathbb{R})}(1-\epsilon_E^2)\leq\int_{\mathbb{R}}|f(x)|^2\chi_{E}(x)dx.$$
Since $f\in L^\infty(\mathbb{R}),$ so 
\begin{eqnarray*}
\int_{\mathbb{R}}\chi_{E}(x)|f(x)|^2dx
&\leq & m(E)\|f\|^2_{L^\infty(\mathbb{R})}.
\end{eqnarray*}
Thus
\begin{equation}
\|f\|^2_{L^\infty(\mathbb{R})}m(E)\geq(1-\epsilon_{E}^2)\|f\|^2_{L^2{(\mathbb{R}})}.
\end{equation}
Therefore,
$$|\Omega|m(E)\|\psi\|^2_{L^2(\mathbb{R})}\|f\|^2_{L^\infty(\mathbb{R})}\geq 2\pi|B|C_{\psi,M}(1-\epsilon_\Omega^2)(1-\epsilon_E^2)\|f\|^2_{L^2(\mathbb{R})}.$$
The proof is complete.
\end{proof}
\begin{theorem}(\textbf{Lieb's uncertainty principle}).
Let $0\leq\epsilon<1,$ $\psi$ is an ALCW and a non-zero$f\in L^2(\mathbb{R}).$ Also let $W^M_{\psi} f$ is $\epsilon-$concentrated on $\Omega\subset\mathbb{R^+}\times\mathbb{R},$ then 
\begin{equation}
|\Omega|\|\psi\|^2_{L^2(\mathbb{R})}\geq 2\pi|B|C_{\psi,M}(1-\epsilon^2)^{\frac{p}{p-2}},~p>2.
\end{equation}
\end{theorem}
\begin{proof}
Since $W_\psi^M f$ is $\epsilon-$concentrated on $\Omega,$ we have 
$$\|\chi_{{\Omega}^c}W_\psi^M f\|_{L^2(\mathbb{R}^+\times\mathbb{R},dadb)}\leq \epsilon^2\|W_\psi^M f\|_{L^2(\mathbb{R}^+\times\mathbb{R},dadb)}.$$
Now,
\begin{eqnarray*}
\|\chi_{{\Omega}^c}W_\psi^M f\|_{L^2(\mathbb{R}^+\times\mathbb{R},dadb)}
%&=&\|\chi_{\Omega}W_\psi^M f\|_{L^2(\mathbb{R}^+\times\mathbb{R},dadb)}+\|\chi_{{\Omega}^c}W_\psi^M f\|_{L^2(\mathbb{R}^+\times\mathbb{R},dadb)}\\
&\leq &\|\chi_{\Omega}W_\psi^M f\|_{L^2(\mathbb{R}^+\times\mathbb{R},dadb)}+\epsilon^2\|W_\psi^M f\|_{L^2(\mathbb{R}^+\times\mathbb{R},dadb)}.
\end{eqnarray*}
This implies,
$$2\pi C_{\psi,M}\|f\|_{L^2(\mathbb{R})}(1-\epsilon^2)\leq \int_{\mathbb{R}^+\times\mathbb{R}}\chi_\Omega (a,b)|(W_\psi^M f)(a,b)|^2dadb.$$
By  Holder's inequality, we get
\begin{eqnarray*}
2\pi C_{\psi,M}\|f\|_{L^2(\mathbb{R})}(1-\epsilon^2)&\leq & \left(\int_{\mathbb{R}^+\times\mathbb{R}}(\chi_\Omega (a,b))^\frac{p}{p-2} dadb\right)^{\frac{p-2}{p}}\left(\int_{\mathbb{R}^+\times\mathbb{R}}|(W_\psi^M f)(a,b)|^p dadb\right)^\frac{2}{p}\\
&\leq &|\Omega|^{\frac{p-2}{p}}\|W^M_\psi f\|^2_{L^p(\mathbb{R}^+\times\mathbb{R},dadb)}.
\end{eqnarray*}
Using equation (\ref{P3LIp}), we get

%$$\|W^M_\psi f\|^2_{L^p(\mathbb{R}^+\times\mathbb{R},dadb)}\leq (2\pi|B|)^{\frac{1}{p}}C_{\psi,M}^{\frac{1}{p}}\|f\|_{L^2(\mathbb{R})}\|\psi\|^{1-\frac{2}{p}}_{L^2(\mathbb{R})},$$

$$(2\pi|B|)^{1-\frac{2}{p}}C_{\psi,M}^{1-\frac{2}{p}}(1-\epsilon^2)\leq |\Omega|^{\frac{p-2}{p}}\|\psi\|^{2-\frac{4}{p}}_{L^2(\mathbb{R})}.$$
Therefore, we get 
$$|\Omega|\|\psi\|^2_{L^2(\mathbb{R})}\geq 2\pi|B|C_{\psi,M}(1-\epsilon^2)^{\frac{p}{p-2}}.$$ 
The proof is complete.
\end{proof}
\begin{corollary}
If $f\in L^2(\mathbb{R})\cap L^4(\mathbb{R}),$ in $L^2(\mathbb{R})-$norm, is $\epsilon_{E}-$ concentrated on $E\subset\mathbb{R}$ and $W^M_{\psi}f$ is $\epsilon_{\Omega}-$concentrated on $\Omega\subset\mathbb{R^+}\times\mathbb{R},$ then 
$$|\Omega|m(E)\|\psi\|^2_{L^2(\mathbb{R})}\|f\|^4_{L^4(\mathbb{R})}\geq 2\pi|B|C_{\psi,M}(1-\epsilon_\Omega^2)^{\frac{p}{p-2}}(1-\epsilon_E^2)^2\|f\|^4_{L^2(\mathbb{R})},~p>2.$$
\end{corollary}
\begin{proof}
The proof follows similarly as theorem \ref{P3D1}, we use Lieb's UP instead of the Donoho-Stark's UP.
\end{proof}
\begin{corollary}
If$f\in L^2(\mathbb{R})\cap L^\infty(\mathbb{R}),$ in $L^2(\mathbb{R})-$norm, is $\epsilon_{E}-$ concentrated on $E\subset\mathbb{R}$ and $W^M_{\psi}f$ is $\epsilon_{\Omega}-$concentrated on $\Omega\subset\mathbb{R^+}\times\mathbb{R},$ then 
$$|\Omega|m(E)\|\psi\|^2_{L^2(\mathbb{R})}\|f\|^2_{L^\infty(\mathbb{R})}\geq 2\pi|B|C_{\psi,M}(1-\epsilon_\Omega^2)^{\frac{p}{p-2}}(1-\epsilon_E^2)\|f\|^2_{L^2(\mathbb{R})},~p>2.$$
\end{corollary}
\begin{proof}
The proof follows similarly as theorem \ref{P3D2}, we use Lieb's UP instead of the Donoho-Stark's UP.
\end{proof}
\section{Orthonormal sequences and uncertainty principle}
We now express the UP in term of the generalized dispersion of $W_{\psi}^M,$ which is defined by
\begin{equation}
\rho_{p}\left(W_{\psi}^Mf\right)=\left(\int_{\mathbb{R}^+\times\mathbb{R}}|(a,b)|^p|(W_{\psi}^Mf)(a,b)|^2dadb\right)^{\frac{1}{p}},
\end{equation}
where $|(a,b)|=\sqrt{a^2+b^2},~p>0.$ 

\begin{definition}
Let $T$ be a bounded linear operator on a Hilbert space $\mathbb{H}$ over the field $\mathbb{F}$ (where $\mathbb{F}$ is $\mathbb{R}$ or $\mathbb{C}$) and $\{u_{n}\}_{n\in\mathbb{N}}$ be an orthonormal basis of $\mathbb{H},$ then $T$ is called a Hilbert-Schmidt operator if 
$$\|T\|_{HS}=\left(\sum_{n=1}^\infty\|Tu_{n}\|^2\right)^{\frac{1}{2}}<\infty.$$ 
\end{definition}
Before discussing the main result of this section, we estimate the Hilbert-Schmidt norm of the product of some orthogonal projection operators and use it to estimate the concentration of  $W_{\psi}^M$ on subset of $\mathbb{R}^+\times\mathbb{R}.$ Similar results were first studied by Wilczok \cite{wilczok2000new} in the case of windowed FT and WT. 
\begin{theorem}
Let $f\in L^2(\mathbb{R}),$ $\psi$ is an ALCW and $\Omega\subset\mathbb{R}^+\times\mathbb{R}$ such that  $|\Omega|<\frac{2\pi |B|C_{\psi,M}}{\|\psi\|^2_{L^2(\mathbb{R})}}.$ Then
$$\|\chi_{\Omega^c}W_{\psi}^Mf\|_{L^2(\mathbb{R}^+\times\mathbb{R})}\geq\sqrt{2\pi|B|C_{\psi,M}-|\Omega|\|\psi\|^2_{L^2(\mathbb{R})}}\|f\|_{L^2(\mathbb{R})}.$$
\end{theorem}
\begin{proof}
We consider the orthogonal projections $P_{\psi}$ from $L^2(\mathbb{R}^+\times\mathbb{R},dadb)$ on the RKHS $W_{\psi}^M(L^2(\mathbb{R}))$ and $P_{\Omega}$ on $L^2(\mathbb{R}^+\times\mathbb{R},dadb)$ defined by $P_{\Omega}F=\chi_{\Omega}F,~\mbox{for all}~F\in L^2(\mathbb{R}^+\times\mathbb{R},dadb).$
According to Saitoh \cite{saitoh1988theory}, for every $(a,b)\in\mathbb{R}^+\times\mathbb{R}$ and $F\in L^2(\mathbb{R}^+\times\mathbb{R},dadb),$ we get
\begin{eqnarray*}
\left(P_{\Omega}P_{\psi}F\right)(a,b)&=&\chi_{\Omega}(a,b)\left\langle F,K_{\psi}^M\left(\cdot,\cdot;a,b\right)\right\rangle_{L^2(\mathbb{R}^+\times\mathbb{R},dadb)}\\
&=&\int_{\mathbb{R}^+\times\mathbb{R}}\chi_{\Omega}(a,b)F(x,y)\overline{K_{\psi}^M(x,y;a,b)}dxdy.
\end{eqnarray*}
Thus the integral operator $P_{\Omega}P_{\psi}$ has the kernel $\mathcal{N}_{\psi,\Omega}^{M}$ defined on $(\mathbb{R}^+\times\mathbb{R})^2$ by
$$\mathcal{N}_{\psi,\Omega}^{M}(x,y;a,b)=F(x,y)\overline{K_{\psi}^M(x,y;a,b)}$$
such that
\begin{align*}
\int_{\mathbb{R}^+\times\mathbb{R}}\int_{\mathbb{R}^+\times\mathbb{R}}&|\mathcal{N}_{\psi,\Omega}^{M}(x,y;a,b)|^2dxdydadb\\
 &= \int_{\mathbb{R}^+\times\mathbb{R}}\left(\int_{\mathbb{R}^+\times\mathbb{R}}|K_{\psi}^M(x,y;a,b)|^2dxdy\right)|\chi_{\Omega}(a,b)|^2dadb\\
&=\int_{\mathbb{R}^+\times\mathbb{R}}\chi_{\Omega}(a,b)\|K_{\psi}^M(\cdot,\cdot;a,b)\|^2_{L^2(\mathbb{R}^+\times\mathbb{R},dadb)}dadb\\
&=\frac{|\Omega|\|\psi\|^2_{L^2(\mathbb{R})}}{2\pi |B|C_{\psi,M}}.
\end{align*}
Now,
 $$\chi_{\Omega}W_{\psi}^Mf=P_{\Omega}P_{\psi}(W_{\psi}^Mf)$$
 implies
 $$\|\chi_{\Omega}W_{\psi}^Mf\|^2_{L^2(\mathbb{R}^+\times\mathbb{R})}\leq\|P_{\Omega}P_{\psi}\|^2\|W_{\psi}^Mf\|^2_{L^2(\mathbb{R}^+\times\mathbb{R})}.$$
 Therefore
 $$\|W_{\psi}^Mf\|^2_{L^2(\mathbb{R}^+\times\mathbb{R})}\leq \|P_{\Omega}P_{\psi}\|^2\|W_{\psi}^Mf\|^2_{L^2(\mathbb{R}^+\times\mathbb{R})}+\|\chi_{\Omega^c}W_{\psi}^Mf\|^2_{L^2(\mathbb{R}^+\times\mathbb{R})}$$
 $$\mbox{i.e.,}~\|\chi_{\Omega^c}W_{\psi}^Mf\|^2_{L^2(\mathbb{R}^+\times\mathbb{R})}\geq\left(1-\|P_{\Omega}P_{\psi}\|^2\right)2\pi|B|C_{\psi,M}\|f\|^2_{L^2(\mathbb{R})}.$$
 Now using the fact that $\|P_{\Omega}P_{\psi}\|\leq\|P_{\Omega}P_{\psi}\|_{HS},$ where $\|\cdot\|$ denotes the operator norm, we obtain
 \begin{eqnarray*}
 \|\chi_{\Omega^c}W_{\psi}^Mf\|^2_{L^2(\mathbb{R}^+\times\mathbb{R})}\geq\left(1-\|P_{\Omega}P_{\psi}\|_{HS}^2\right)2\pi|B|C_{\psi,M}\|f\|^2_{L^2(\mathbb{R})}.
 \end{eqnarray*}
 Hence, we obtain
 $$\|\chi_{\Omega^c}W_{\psi}^Mf\|_{L^2(\mathbb{R}^+\times\mathbb{R})}\geq\sqrt{2\pi|B|C_{\psi,M}-|\Omega|\|\psi\|^2_{L^2(\mathbb{R})}}\|f\|_{L^2(\mathbb{R})}.$$
This proves the theorem.
\end{proof}

\begin{theorem}
If $\psi$ is an ALCW, $\{\phi_{n}\}_{n\in\mathbb{N}}\subset L^2(\mathbb{R})$ be an ONS and $\Omega\subset\mathbb{R}^+\times\mathbb{R}$ be such that its measure $|\Omega|<\infty,$ then for any non-empty $\wedge\subset\mathbb{N},$  
\begin{equation}
\sum_{n\in\wedge}\left(1-\left\|\chi_{\Omega^c}W_\psi^M\left(\frac{\phi_{n}}{\sqrt{2\pi|B|C_{\psi,M}}}\right)\right\|_{L^2(\mathbb{R^+}\times\mathbb{R},dadb)}\right)\leq\frac{|\Omega|\|\psi\|_{L^2(\mathbb{R})}^2}{2\pi|B|C_{\psi,M}}.
\end{equation}
\end{theorem}
\begin{proof}
Consider the orthonormal basis $\{h_{n}\}_{n\in\mathbb{N}}$ of $L^2(\mathbb{R^+}\times\mathbb{R},dadb).$ It has been proved in the above theorem that $P_{\Omega}P_{\psi}$ is a Hilbert-Schmidt operator such that $\|P_{\Omega}P_{\psi}\|^2_{HS}=\frac{|\Omega|\|\psi\|_{L^2(\mathbb{R})}^2}{2\pi|B|C_{\psi,M}}.$\\
Since $P_{\Omega}^2=P_{\Omega}$ and both $P_{\psi},~P_{\Omega}$ are self-adjoint, the operator $T=(P_{\Omega}P_{\psi})^{\star}(P_{\Omega}P_{\psi})=P_{\psi}P_{\Omega}P_{\psi}$ is positive and is such that 
\begin{eqnarray*}
\sum_{n\in\mathbb{N}}\langle Th_{n},h_{n}\rangle_{L^2(\mathbb{R}^+\times\mathbb{R},dadb)}&=&\sum_{n\in\mathbb{N}}\langle P_{\Omega}P_{\psi}h_{n},P_{\Omega}P_{\psi}h_{n}\rangle_{L^2(\mathbb{R}^+\times\mathbb{R},dadb)}\\
&=&\sum_{n\in\mathbb{N}}\|P_{\Omega}P_{\psi}h_{n}\|^2_{L^2(\mathbb{R}^+\times\mathbb{R},dadb)}\\
&=&\|P_{\Omega}P_{\psi}\|^2_{HS}\\
&=&\frac{|\Omega|\|\psi\|_{L^2(\mathbb{R})}^2}{2\pi|B|C_{\psi,M}}<\infty.
\end{eqnarray*}
Therefore, $T$ is a trace class operator with $Tr(T)=\frac{|\Omega|\|\psi\|_{L^2(\mathbb{R})}^2}{2\pi|B|C_{\psi,M}}.$\\
Now as $\{\phi_{n}\}_{n\in\mathbb{N}}$ is an ONS, from equation (\ref{P3IPRWT}), it follows that $\left\{W_\psi^M\left(\frac{\phi_{n}}{\sqrt{2\pi|B|C_{\psi,M}}}\right)\right\}_{n\in\mathbb{N}}$ is an ONS in $L^2(\mathbb{R}^+\times\mathbb{R},dadb).$\\
Hence, we have
\begin{align*}
\sum_{n\in\wedge}&\left\langle P_{\Omega}W_{\psi}^M\left(\frac{\phi_{n}}{\sqrt{2\pi|B|C_{\psi,M}}}\right),W_{\psi}^M\left(\frac{\phi_{n}}{\sqrt{2\pi|B|C_{\psi,M}}}\right)\right\rangle_{L^2(\mathbb{R}^+\times\mathbb{R},dadb)}\\
&=\sum_{n\in\wedge}\left\langle P_{\psi}P_{\Omega}P_{\psi}\left(W_{\psi}^M\left(\frac{\phi_{n}}{\sqrt{2\pi|B|C_{\psi,M}}}\right)\right),W_{\psi}^M\left(\frac{\phi_{n}}{\sqrt{2\pi|B|C_{\psi,M}}}\right)\right\rangle_{L^2(\mathbb{R}^+\times\mathbb{R},dadb)}\\
&=\sum_{n\in\wedge}\left\langle T\left(W_{\psi}^M\left(\frac{\phi_{n}}{\sqrt{2\pi|B|C_{\psi,M}}}\right)\right),W_{\psi}^M\left(\frac{\phi_{n}}{\sqrt{2\pi|B|C_{\psi,M}}}\right)\right\rangle_{L^2(\mathbb{R}^+\times\mathbb{R},dadb)}\\
&\leq \sum_{n\in\mathbb{N}}\left\langle T\left(W_{\psi}^M\left(\frac{\phi_{n}}{\sqrt{2\pi|B|C_{\psi,M}}}\right)\right),W_{\psi}^M\left(\frac{\phi_{n}}{\sqrt{2\pi|B|C_{\psi,M}}}\right)\right\rangle_{L^2(\mathbb{R}^+\times\mathbb{R},dadb)}\\
&= Tr(T)=\frac{|\Omega|\|\psi\|_{L^2(\mathbb{R})}^2}{2\pi|B|C_{\psi,M}}.
\end{align*}
For each $n\in\wedge,$ we have 
\begin{align*}
&\left\langle P_{\Omega}W_{\psi}^M\left(\frac{\phi_{n}}{\sqrt{2\pi|B|C_{\psi,M}}}\right),W_{\psi}^M\left(\frac{\phi_{n}}{\sqrt{2\pi|B|C_{\psi,M}}}\right)\right\rangle_{L^2(\mathbb{R}^+\times\mathbb{R},dadb)}\\
&=\left\langle \chi_{\Omega}W_{\psi}^M\left(\frac{\phi_{n}}{\sqrt{2\pi|B|C_{\psi,M}}}\right),W_{\psi}^M\left(\frac{\phi_{n}}{\sqrt{2\pi|B|C_{\psi,M}}}\right)\right\rangle_{L^2(\mathbb{R}^+\times\mathbb{R},dadb)}\\
&=1-\left\langle \chi_{\Omega^c}W_{\psi}^M\left(\frac{\phi_{n}}{\sqrt{2\pi|B|C_{\psi,M}}}\right),W_{\psi}^M\left(\frac{\phi_{n}}{\sqrt{2\pi|B|C_{\psi,M}}}\right)\right\rangle_{L^2(\mathbb{R}^+\times\mathbb{R},dadb)}\\
&\geq 1-\left\|\chi_{\Omega^c}W_{\psi}^M\left(\frac{\phi_{n}}{\sqrt{2\pi|B|C_{\psi,M}}}\right)\right\|_{L^2(\mathbb{R}^+\times\mathbb{R},dadb)}.
\end{align*}
Thus, we have 
$$\sum_{n\in\wedge}\left(1-\left\|\chi_{\Omega^c}W_\psi^M\left(\frac{\phi_{n}}{\sqrt{2\pi|B|C_{\psi,M}}}\right)\right\|_{L^2(\mathbb{R^+}\times\mathbb{R},dadb)}\right)\leq\frac{|\Omega|\|\psi\|_{L^2(\mathbb{R})}^2}{2\pi|B|C_{\psi,M}}.$$
This proves the theorem                                                                                                                                                                                      .
\end{proof}
The theorem below shows that, if the LCWT of each member of an ONS are $\epsilon-$concentrated in a set of finite measure then the sequence is necessarily finite. The theorem also gives an upper bound of the  cardinality of the so proved finite sequence.
\begin{theorem}\label{P3theoCard}
Let $s,\epsilon>0$ such that $\epsilon<1.$ Let $G_s=\{(a,b)\in\mathbb{R}^+\times\mathbb{R}:a^2+b^2\leq s^2\}$ and $\psi$ is an ALCW. Also let $\wedge\subset \mathbb{N}$ be non-empty  and $\{\phi_{n}\}_{n\in\wedge}\subset L^2(\mathbb{R})$ be an ONS. If $W_{\psi}^M\left(\frac{\phi_{n}}{\sqrt{2\pi|B|C_{\psi,M}}}\right)$ is $\epsilon-$concentrated in $G_s$ for all $n\in\wedge,$ then $\wedge$ is finite and 
\begin{equation}
Card(\wedge)\leq\frac{s^2\|\psi\|^2_{L^2(\mathbb{R})}}{4|B|C_{\psi,M}(1-\epsilon)},
\end{equation}
where $Card(\wedge)$ denotes the cardinality of $\wedge.$
\end{theorem}
\begin{proof}
Applying above theorem, we have
$$\sum_{n\in\wedge}\left(1-\left\|\chi_{G_s^c}W_\psi^M\left(\frac{\phi_{n}}{\sqrt{2\pi|B|C_{\psi,M}}}\right)\right\|_{L^2(\mathbb{R^+}\times\mathbb{R},dadb)}\right)\leq\frac{|G_s|\|\psi\|_{L^2(\mathbb{R})}^2}{2\pi|B|C_{\psi,M}}.$$
Again, since for each $W_{\psi}^M\left(\frac{\phi_{n}}{\sqrt{2\pi|B|C_{\psi,M}}}\right)$ is $\epsilon-$concentrated in $G_s,$ we have
$$\left\|\chi_{G_s^c}W_{\psi}^M\left(\frac{\phi_{n}}{\sqrt{2\pi|B|C_{\psi,M}}}\right)\right\|_{L^2(\mathbb{R}^+\times\mathbb{R})}\leq\epsilon.$$
Therefore, it follows that 
$$\sum_{n\in\wedge}(1-\epsilon)\leq\frac{|G_s|\|\psi\|_{L^2(\mathbb{R})}^2}{2\pi|B|C_{\psi,M}}$$
$$\mbox{i.e.,}~Card(\wedge)(1-\epsilon)\leq\frac{|G_s|\|\psi\|_{L^2(\mathbb{R})}^2}{2\pi|B|C_{\psi,M}}.$$
Thus $Card(\wedge)$ is finite and using $|G_s|=\frac{\pi s^2}{2},$ we obtain
$$Card(\wedge)\leq\frac{s^2\|\psi\|_{L^2(\mathbb{R})}^2}{4|B|(1-\epsilon)C_{\psi,M}}.$$
The proof is complete.
\end{proof}
\begin{corollary}
Let $p>0,$ $R>0$ and $\psi$ is an ALCW. Also let $\wedge\subset \mathbb{N},$ be non-empty and $\{\phi_{n}\}_{n\in\wedge}\subset L^2(\mathbb{R})$ be an ONS. Then $\wedge$ is finite if $\left\{\rho_{p}\left(W_{\psi}^M\left(\frac{\phi_{n}}{\sqrt{2\pi|B|C_{\psi,M}}}\right)\right)\right\}_{n\in\wedge}$ is uniformly bounded. Moreover, if it is uniformly bounded by $R,$ then
$$Card(\wedge)\leq\frac{2^{\frac{4}{p}+1}R^2\|\psi\|^2_{L^2(\mathbb{R})}}{4|B|C_{\psi,M}}.$$
\end{corollary}
\begin{proof}
Since for each $n\in\wedge,$ $\rho_{p}\left(W_{\psi}^M\left(\frac{\phi_{n}}{\sqrt{2\pi|B|C_{\psi,M}}}\right)\right)\leq R,$ and thus
\begin{align*}
&\int_{|(a,b)|\geq R2^{\frac{2}{p}}}\left|\left(W_{\psi}^M\left(\frac{\phi_{n}}{\sqrt{2\pi|B|C_{\psi,M}}}\right)\right)(a,b)\right|^2dadb\\
&=\int_{|(a,b)|\geq R2^{\frac{2}{p}}}|(a,b)|^{-p}|(a,b)|^p\left|\left(W_{\psi}^M\left(\frac{\phi_{n}}{\sqrt{2\pi|B|C_{\psi,M}}}\right)\right)(a,b)\right|^2dadb\\
&\leq \frac{1}{\left(R2^{\frac{2}{p}}\right)^p}\int_{\mathbb{R}^+\times\mathbb{R}}|(a,b)|^p\left|\left(W_{\psi}^M\left(\frac{\phi_{n}}{\sqrt{2\pi|B|C_{\psi,M}}}\right)\right)(a,b)\right|^2dadb\\
&\leq  \frac{1}{4}.
\end{align*}
Thus it follows that, for each $n\in\wedge,$ $W_{\psi}^M\left(\frac{\phi_{n}}{\sqrt{2\pi|B|C_{\psi,M}}}\right)$ is $\frac{1}{2}-$concentrated in
$$G_{R2^{\frac{2}{p}}}=\left\{(a,b)\in\mathbb{R}^+\times\mathbb{R}:|(a,b)|<R2^{\frac{2}{p}}\right\}.$$
Thus, from theorem \ref{P3theoCard}, it follows that $\wedge$ is finite and
$$Card(\wedge)\leq\frac{\left(R2^{\frac{2}{p}}\right)^2\|\psi\|^2_{L^2(\mathbb{R})}}{4|B|\left(1-\frac{1}{2}\right)C_{\psi,M}},$$
$$\mbox{i.e.,}~Card(\wedge)\leq\frac{2^{\frac{4}{p}+1}R^2\|\psi\|^2_{L^2(\mathbb{R})}}{4|B|C_{\psi,M}}.$$
Thus the proof is complete.
\end{proof}
\begin{lemma}\label{P3lemma5.2}
Let $p>0,$ $\psi$ is an ALCW and $\{\phi_{n}\}_{n\in\mathbb{N}}\subset L^2(\mathbb{R})$ be an ONS, then $\exists$ $m_{0}\in\mathbb{Z}$ for which
$$\rho_{p}\left(W_{\psi}^M\left(\frac{\phi_{n}}{\sqrt{2\pi|B|C_{\psi,M}}}\right)\right)\geq 2^{m_{0}},~\forall~n\in\mathbb{N}.$$
\end{lemma}
\begin{proof}
Define $P_{m}=\left\{n\in\mathbb{N}:\rho_{p}\left(W_{\psi}^M\left(\frac{\phi_{n}}{\sqrt{2\pi|B|C_{\psi,M}}}\right)\right)\in [2^{m-1},2^m)\right\},$ for each $m\in\mathbb{Z}.$\\
Then for each $n\in P_{m},$ we get
$$\int_{\mathbb{R}^+\times\mathbb{R}}|(a,b)|^{p}\left|\left(W_{\psi}^M\left(\frac{\phi_{n}}{\sqrt{2\pi|B|C_{\psi,M}}}\right)\right)(a,b)\right|^2dadb<2^{mp}.$$
Now,
\begin{align*}
&\int_{|(a,b)|\geq 2^{m+\frac{2}{p}}}\left|\left(W_{\psi}^M\left(\frac{\phi_{n}}{\sqrt{2\pi|B|C_{\psi,M}}}\right)\right)(a,b)\right|^2dadb\\
%&=\int_{|(a,b)|\geq 2^{m+\frac{2}{p}}}|(a,b)|^{-p}|(a,b)|^p\left|\left(W_{\psi}^M\left(\frac{\phi_{n}}{\sqrt{2\pi|B|C_{\psi,M}}}\right)\right)(a,b)\right|^2dadb \\
&\leq \frac{1}{2^{mp+2}}\int_{\mathbb{R}^+\times\mathbb{R}}|(a,b)|^p\left|\left(W_{\psi}^M\left(\frac{\phi_{n}}{\sqrt{2\pi|B|C_{\psi,M}}}\right)\right)(a,b)\right|^2dadb\\
&\leq \frac{1}{2^{mp+2}}\left\{\rho_{p}\left(W_{\psi}^M\left(\frac{\phi_{n}}{\sqrt{2\pi|B|C_{\psi,M}}}\right)\right)\right\}^p.
\end{align*}
This gives 
$$\int_{|(a,b)|\geq 2^{\frac{2}{p}+m}}\left|\left(W_{\psi}^M\left(\frac{\phi_{n}}{\sqrt{2\pi|B|C_{\psi,M}}}\right)\right)(a,b)\right|^2dadb\leq\frac{1}{4}.$$
Thus it follows that, for each $n\in P_{m},$ $W_{\psi}^M\left(\frac{\phi_{n}}{\sqrt{2\pi|B|C_{\psi,M}}}\right)$ is $\frac{1}{2}-$concentrated on the set 
$$G_{2^{m+\frac{2}{p}}}=\left\{(a,b)\in\mathbb{R}^+\times\mathbb{R}:|(a,b)|<2^{m+\frac{2}{p}}\right\}.$$
Therefore, $P_m$ is finite and 
$$Card(P_m)\leq\frac{2^{2m+\frac{4}{p}+1}\|\psi\|^2_{L^2(\mathbb{R})}}{4|B|C_{\psi,M}},~\mbox{for all}~m\in\mathbb{Z}.$$
Letting $m\to -\infty,$ we get $$\lim_{m \to -\infty} Card(P_{m})=0.$$
Hence $\exists$ $m_{0}\in\mathbb{Z}$ such that for all $m<m_{0},$ $P_{m}$ are empty sets.
Therefore, $\rho_{p}\left(W_{\psi}^M\left(\frac{\phi_{n}}{\sqrt{2\pi|B|C_{\psi,M}}}\right)\right)\geq 2^{m_{0}},~\mbox{for all}~n\in\mathbb{N}.$ 
\end{proof}
\begin{theorem}(\textbf{Shapiro's Dispersion theorem}). 
Let $\psi$ be an ALCW and $\{\phi_{n}\}_{n\in\mathbb{N}}\subset L^2(\mathbb{R})$ be an ONS, then for every $p>0$ and non-empty finite $\wedge\subset\mathbb{N},$
$$\sum_{n\in\wedge}\left\{\rho_{p}\left(W_{\psi}^M\left(\frac{\phi_{n}}{\sqrt{2\pi|B|C_{\psi,M}}}\right)\right)\right\}^p\geq \frac{(Card(\wedge))^{\frac{p}{2}+1}}{2^{p+1}}\left(\frac{3|B|C_{\psi,M}}{2^{\frac{4}{p}+2}\|\psi\|^2_{L^2(\mathbb{R})}}\right)^{\frac{p}{2}}.$$
\end{theorem}
\begin{proof}
Let $m_{0}$ be an integer defined in the above lemma. Let $k\in\mathbb{Z}$ such that $k\geq m_{0}.$ Define $Q_{k}=\bigcup\limits_{m=m_{0}}^k P_{m}.$ Then we have
\begin{eqnarray*}
Card(Q_{k})=\sum_{m=m_{0}}^k Card(P_{m})&\leq&\sum_{m=m_{0}}^k\frac{2^{2m+\frac{4}{p}+1}\|\psi\|^2_{L^2(\mathbb{R})}}{4|B|C_{\psi,M}}\\
&=&\frac{2^{\frac{4}{p}+1}\|\psi\|^2_{L^2(\mathbb{R})}}{4|B|C_{\psi,M}}\sum_{m=m_{0}}^k 2^{2m}\\
&\leq&\frac{2^{\frac{4}{p}+1}\|\psi\|^2_{L^2(\mathbb{R})}}{4|B|C_{\psi,M}} \frac{2^{2k+2}}{3}
\end{eqnarray*}
$$
\mbox{i.e.,}~Card(Q_{k})\leq\frac{2^{\frac{4}{p}+1}\|\psi\|^2_{L^2(\mathbb{R})}}{3|B|C_{\psi,M}} 2^{2k}.$$
Let $C=\frac{2^{\frac{4}{p}+2}\|\psi\|^2_{L^2(\mathbb{R})}}{3|B|C_{\psi,M}}.$ Then $Card(Q_{k})\leq\frac{C}{2}2^{2k}.$
If $Card(\wedge)>2^{2(m_{0}+1)},$ then $\frac{1}{2\log{2}}\log\left(\frac{Card(\wedge)}{C}\right)>m_{0}+1.$\\
Let us choose an integer $k>m_{0}+1$ such that 
$$k-1<\frac{1}{2\log{2}}\log(\frac{Card(\wedge)}{C})\leq k.$$
Then it results in 
$$C2^{2(k-1)}<Card(\wedge)\leq C2^{2k}.$$
Thus, we have
$$Card(Q_{k-1})=\frac{C}{2}2^{2(k-1)}<\frac{Card(\wedge)}{2}.$$
Therefore,
\begin{eqnarray*}
\sum_{n\in\wedge}\left\{\rho_{p}\left(W_{\psi}^M\left(\frac{\phi_{n}}{\sqrt{2\pi|B|C_{\psi,M}}}\right)\right)\right\}^p&\geq&\sum_{n\not\in Q_{k-1}} \left\{\rho_{p}\left(W_{\psi}^M\left(\frac{\phi_{n}}{\sqrt{2\pi|B|C_{\psi,M}}}\right)\right)\right\}^p\\
&\geq& \frac{Card(\wedge)}{2}2^{(k-1)p}\\
&=&\frac{Card(\wedge)}{2.2^{p}}2^{kp}.
\end{eqnarray*}
Since, $Card(\wedge)\leq C2^{2k},$ we have $\left(\frac{Card(\wedge)}{C}\right)^{\frac{p}{2}}\leq 2^{kp}.$\\
Therefore, 
$$\sum_{n\in\wedge}\left\{\rho_{p}\left(W_{\psi}^M\left(\frac{\phi_{n}}{\sqrt{2\pi|B|C_{\psi,M}}}\right)\right)\right\}^p \geq \frac{(Card(\wedge))^{\frac{p}{2}+1}}{2^{p+1}}\left(\frac{1}{C}\right)^{\frac{p}{2}}.$$
Again, if $Card(\wedge)\leq2^{2(m_{0}+1)},$ then 
$$\sum_{n\in\wedge}\left\{\rho_{p}\left(W_{\psi}^M\left(\frac{\phi_{n}}{\sqrt{2\pi|B|C_{\psi,M}}}\right)\right)\right\}^p \geq Card(\wedge)2^{m_{0}p},~\mbox{(using lemma \ref{P3lemma5.2})}.$$
Now, $Card(\wedge)\leq C2^{2(m_{0}+1)}$ implies $\frac{1}{2^p}\left(\frac{Card(\wedge}{C})\right)^{\frac{p}{2}}\leq 2^{m_{0}p}.$
Thus we have
$$\sum_{n\in\wedge}\left\{\rho_{p}\left(W_{\psi}^M\left(\frac{\phi_{n}}{\sqrt{2\pi|B|C_{\psi,M}}}\right)\right)\right\}^p \geq \frac{(Card(\wedge))^{\frac{p}{2}+1}}{2^{p}}\left(\frac{1}{C}\right)^{\frac{p}{2}}.$$
Hence, for any non-empty finite $\wedge\subset\mathbb{N},$ we have
$$\sum_{n\in\wedge}\left\{\rho_{p}\left(W_{\psi}^M\left(\frac{\phi_{n}}{\sqrt{2\pi|B|C_{\psi,M}}}\right)\right)\right\}^p \geq \frac{(Card(\wedge))^{\frac{p}{2}+1}}{2^{p+1}}\left(\frac{1}{C}\right)^{\frac{p}{2}}.$$
Therefore, putting the value of $C$ we get
$$\sum_{n\in\wedge}\left\{\rho_{p}\left(W_{\psi}^M\left(\frac{\phi_{n}}{\sqrt{2\pi|B|C_{\psi,M}}}\right)\right)\right\}^p \geq \frac{(Card(\wedge))^{\frac{p}{2}+1}}{2^{p+1}}\left(\frac{3|B|C_{\psi,M}}{2^{\frac{4}{p}+2}\|\psi\|^2_{L^2(\mathbb{R})}}\right)^{\frac{p}{2}}.$$
This completes the proof.
\end{proof}
\section{Conclusions}
We have proposed a novel time-frequency analyzing tool, namely LCWT, which combines the advantages of the LCT and the WT and offers time and linear canonical domain spectral information simultaneously in the time-LCT-frequency plane. We have studied its properties like inner product relation, reconstruction formula and also characterized its range. We also gave a lower bound of the measure of essential support of the LCWT via UP of Donoho-Stark and Lieb. Finally, we have studied the Shapiro's mean dispersion theorem associated with the LCWT. 
\section{Acknowledgement}
This work is partially supported by UGC File No. 16-9(June 2017)/2018(NET/CSIR), New Delhi, India. 
\bibliography{P3MasterB3}

\begin{thebibliography}{36}
\providecommand{\natexlab}[1]{#1}
\providecommand{\url}[1]{\texttt{#1}}
\expandafter\ifx\csname urlstyle\endcsname\relax
  \providecommand{\doi}[1]{doi: #1}\else
  \providecommand{\doi}{doi: \begingroup \urlstyle{rm}\Url}\fi

\bibitem[Debnath and Bhatta(2014)]{debnath2014integral}
L.~Debnath and D.~Bhatta.
\newblock \emph{Integral transforms and their applications}.
\newblock CRC press, 2014.

\bibitem[Almeida(1994)]{almeida1994fractional}
L.B. Almeida.
\newblock The fractional {F}ourier transform and time-frequency
  representations.
\newblock \emph{IEEE Transactions on signal processing}, 42\penalty0
  (11):\penalty0 3084--3091, 1994.

\bibitem[Chen et~al.(2021)Chen, Fu, Grafakos, and Wu]{CHEN202171}
Wei Chen, Zunwei Fu, Loukas Grafakos, and Yue Wu.
\newblock Fractional fourier transforms on lp and applications.
\newblock \emph{Applied and Computational Harmonic Analysis}, 55:\penalty0
  71--96, 2021.
\newblock ISSN 1063-5203.
\newblock \doi{https://doi.org/10.1016/j.acha.2021.04.004}.
\newblock URL
  \url{https://www.sciencedirect.com/science/article/pii/S1063520321000385}.

\bibitem[I.~Zayed(1998)]{i1998fractional}
Ahmed I.~Zayed.
\newblock Fractional fourier transform of generalized functions.
\newblock \emph{Integral Transforms and Special Functions}, 7\penalty0
  (3-4):\penalty0 299--312, 1998.

\bibitem[Moshinsky and Quesne(1971)]{moshinsky1971linear}
M.~Moshinsky and C.~Quesne.
\newblock Linear canonical transformations and their unitary representations.
\newblock \emph{Journal of Mathematical Physics}, 12\penalty0 (8):\penalty0
  1772--1780, 1971.

\bibitem[Sharma and Joshi(2006)]{sharma2006signal}
K.K. Sharma and S.D. Joshi.
\newblock Signal separation using linear canonical and fractional {F}ourier
  transforms.
\newblock \emph{Optics communications}, 265\penalty0 (2):\penalty0 454--460,
  2006.

\bibitem[Wei and Li(2014{\natexlab{a}})]{wei2014reconstruction}
D.~Wei and Y.~Li.
\newblock Reconstruction of multidimensional bandlimited signals from
  multichannel samples in linear canonical transform domain.
\newblock \emph{IET Signal Processing}, 8\penalty0 (6):\penalty0 647--657,
  2014{\natexlab{a}}.

\bibitem[Barshan et~al.(1997)Barshan, Kutay, and Ozaktas]{barshan1997optimal}
B.~Barshan, M.A. Kutay, and H.M. Ozaktas.
\newblock Optimal filtering with linear canonical transformations.
\newblock \emph{Optics communications}, 135\penalty0 (1-3):\penalty0 32--36,
  1997.

\bibitem[Gao and Li(2021)]{GAO2021108233}
Wen-Biao Gao and Bing-Zhao Li.
\newblock The octonion linear canonical transform: Definition and properties.
\newblock \emph{Signal Processing}, 188:\penalty0 108233, 2021.
\newblock ISSN 0165-1684.
\newblock \doi{https://doi.org/10.1016/j.sigpro.2021.108233}.
\newblock URL
  \url{https://www.sciencedirect.com/science/article/pii/S016516842100270X}.

\bibitem[Healy et~al.(2015)Healy, Kutay, Ozaktas, and
  Sheridan]{healy2015linear}
J.J. Healy, M.A. Kutay, H.M. Ozaktas, and J.T. Sheridan.
\newblock \emph{Linear canonical transforms: Theory and applications}, volume
  198.
\newblock Springer, 2015.

\bibitem[Debnath and Shah(2002)]{debnath2002wavelet}
L.~Debnath and F.A. Shah.
\newblock \emph{Wavelet transforms and their applications}.
\newblock Springer, 2002.

\bibitem[Dai et~al.(2017)Dai, Zheng, and Wang]{dai2017new}
H.~Dai, Z.~Zheng, and W.~Wang.
\newblock A new fractional wavelet transform.
\newblock \emph{Communications in Nonlinear Science and Numerical Simulation},
  44:\penalty0 19--36, 2017.

\bibitem[Kou and Xu(2012)]{kou2012windowed}
K.I. Kou and R.H. Xu.
\newblock Windowed linear canonical transform and its applications.
\newblock \emph{Signal Processing}, 92\penalty0 (1):\penalty0 179--188, 2012.

\bibitem[Wei and Li(2014{\natexlab{b}})]{wei2014generalized}
D.~Wei and Y.M. Li.
\newblock Generalized wavelet transform based on the convolution operator in
  the linear canonical transform domain.
\newblock \emph{Optik}, 125\penalty0 (16):\penalty0 4491--4496,
  2014{\natexlab{b}}.

\bibitem[Guo and Li(2018)]{guo2018linear}
Y.~Guo and B.Z. Li.
\newblock The linear canonical wavelet transform on some function spaces.
\newblock \emph{International Journal of Wavelets, Multiresolution and
  Information Processing}, 16\penalty0 (01):\penalty0 1850010, 2018.

\bibitem[Shi et~al.(2012)Shi, Zhang, and Liu]{shi2012novel}
J.~Shi, N.T. Zhang, and X.P. Liu.
\newblock A novel fractional wavelet transform and its applications.
\newblock \emph{Science China Information Sciences}, 55\penalty0 (6):\penalty0
  1270--1279, 2012.

\bibitem[Prasad et~al.(2014)Prasad, Manna, Mahato, and
  Singh]{prasad2014generalized}
A.~Prasad, S.~Manna, A.~Mahato, and V.K. Singh.
\newblock The generalized continuous wavelet transform associated with the
  fractional {F}ourier transform.
\newblock \emph{Journal of computational and applied mathematics},
  259:\penalty0 660--671, 2014.

\bibitem[Folland and Sitaram(1997)]{folland1997uncertainty}
G.B. Folland and A.~Sitaram.
\newblock The uncertainty principle: a mathematical survey.
\newblock \emph{Journal of {F}ourier analysis and applications}, 3\penalty0
  (3):\penalty0 207--238, 1997.

\bibitem[Shapiro(1991)]{shapiro1991uncertainty}
H.S. Shapiro.
\newblock Uncertainty principles for bases in ${L}^2(\mathbb{R})$.
\newblock \emph{Unpublished manuscript}, 1991.

\bibitem[Jaming and Powell(2007)]{jaming2007uncertainty}
P.~Jaming and A.M. Powell.
\newblock Uncertainty principles for orthonormal sequences.
\newblock \emph{Journal of Functional Analysis}, 243\penalty0 (2):\penalty0
  611--630, 2007.

\bibitem[Malinnikova(2010)]{malinnikova2010orthonormal}
E.~Malinnikova.
\newblock Orthonormal sequences in ${L}^2(\mathbb{R}^d)$ and time frequency
  localization.
\newblock \emph{Journal of Fourier Analysis and Applications}, 16\penalty0
  (6):\penalty0 983--1006, 2010.

\bibitem[Lamouchi and Omri(2016)]{lamouchi2016time}
H.~Lamouchi and S.~Omri.
\newblock Time-frequency localization for the short time {F}ourier transform.
\newblock \emph{Integral Transforms and Special Functions}, 27\penalty0
  (1):\penalty0 43--54, 2016.

\bibitem[Hamadi and Lamouchi(2017)]{hamadi2017shapiro}
N.B. Hamadi and H.~Lamouchi.
\newblock Shapiro’s uncertainty principle and localization operators
  associated to the continuous wavelet transform.
\newblock \emph{Journal of Pseudo-Differential Operators and Applications},
  8\penalty0 (1):\penalty0 35--53, 2017.

\bibitem[Hamadi and Omri(2018)]{b2018uncertainty}
N.B. Hamadi and S.~Omri.
\newblock Uncertainty principles for the continuous wavelet transform in the
  {H}ankel setting.
\newblock \emph{Applicable Analysis}, 97\penalty0 (4):\penalty0 513--527, 2018.

\bibitem[Hamadi et~al.(2020)Hamadi, Hafirassou, and
  Herch]{hamadi2020uncertainty}
N.B. Hamadi, Z.~Hafirassou, and H.~Herch.
\newblock Uncertainty principles for the {H}ankel--{S}tockwell transform.
\newblock \emph{Journal of Pseudo-Differential Operators and Applications},
  pages 1--22, 2020.

\bibitem[Nefzi(2021)]{nefzi2021shapiro}
B.~Nefzi.
\newblock Shapiro and local uncertainty principles for the multivariate
  continuous shearlet transform.
\newblock \emph{Integral Transforms and Special Functions}, 32\penalty0
  (2):\penalty0 154--173, 2021.

\bibitem[Wilczok(2000)]{wilczok2000new}
E.~Wilczok.
\newblock New uncertainty principles for the continuous {G}abor transform and
  the continuous wavelet transform.
\newblock \emph{Documenta Mathematica}, 5:\penalty0 201--226, 2000.

\bibitem[Verma and Gupta(2021)]{verma2021note}
A.K. Verma and B.~Gupta.
\newblock A note on continuous fractional wavelet transform in $\mathbb{R}^n$.
\newblock \emph{International Journal of Wavelets, Multiresolution and
  Information Processing}, page 2150050, 2021.

\bibitem[Bahri and Ashino(2017)]{bahri2017logarithmic}
M.~Bahri and R.~Ashino.
\newblock Logarithmic uncertainty principle, convolution theorem related to
  continuous fractional wavelet transform and its properties on a generalized
  {S}obolev space.
\newblock \emph{International Journal of Wavelets, Multiresolution and
  Information Processing}, 15\penalty0 (05):\penalty0 1750050, 2017.

\bibitem[Shah and Tantary(2019)]{shah2019non}
F.A. Shah and A.Y. Tantary.
\newblock Non-isotropic angular {S}tockwell transform and the associated
  uncertainty principles.
\newblock \emph{Applicable Analysis}, pages 1--25, 2019.

\bibitem[Huo et~al.(2019)Huo, Sun, and Xiao]{huo2019uncertainty}
H.~Huo, W.~Sun, and L.~Xiao.
\newblock Uncertainty principles associated with the offset linear canonical
  transform.
\newblock \emph{Mathematical Methods in the Applied Sciences}, 42\penalty0
  (2):\penalty0 466--474, 2019.

\bibitem[Daubechies(1992)]{daubechies1992ten}
I.~Daubechies.
\newblock \emph{Ten lectures on wavelets}.
\newblock SIAM, 1992.

\bibitem[Donoho and Stark(1989)]{donoho1989uncertainty}
D.L. Donoho and P.B. Stark.
\newblock Uncertainty principles and signal recovery.
\newblock \emph{SIAM Journal on Applied Mathematics}, 49\penalty0 (3):\penalty0
  906--931, 1989.

\bibitem[Gr{\"o}chenig(2001)]{grochenig2001foundations}
K.~Gr{\"o}chenig.
\newblock \emph{Foundations of time-frequency analysis}.
\newblock Springer Science \& Business Media, 2001.

\bibitem[Kou et~al.(2012)Kou, Xu, and Zhang]{kou2012paley}
K.I. Kou, R.H. Xu, and Y.H. Zhang.
\newblock Paley--{W}iener theorems and uncertainty principles for the windowed
  linear canonical transform.
\newblock \emph{Mathematical Methods in the Applied Sciences}, 35\penalty0
  (17):\penalty0 2122--2132, 2012.

\bibitem[Saitoh(1988)]{saitoh1988theory}
S.~Saitoh.
\newblock Theory of reproducing kernels and its applications.
\newblock \emph{Longman Scientific \& Technical}, 1988.

\end{thebibliography}
\bibliographystyle{unsrtnat}
\end{document}